\documentclass[12pt]{amsart}
\usepackage[latin9]{inputenc}
\usepackage{color}
\usepackage{amsfonts}
\usepackage{amssymb}
\usepackage{amsmath}
\usepackage{amsthm}
\usepackage{geometry}
\usepackage{verbatim}
\usepackage[toc,page]{appendix}
\DeclareMathOperator\Gal{Gal}
\DeclareMathOperator\Aut{Aut}
\DeclareMathOperator\core{core}

\DeclareMathOperator\PSL{PSL}
\DeclareMathOperator\PGL{PGL}
\DeclareMathOperator\GL{GL}
\DeclareMathOperator\Sym{Sym}
\DeclareMathOperator\Stab{Stab}
\DeclareMathOperator\Aff{Aff}
\DeclareMathOperator\Image{Im}
\DeclareMathOperator\HG{H}
\DeclareMathOperator\ZG{Z}
\DeclareMathOperator\tr{tr}
\DeclareMathOperator\res{res}
\DeclareMathOperator\cores{cor}

\begin{document}
\providecommand{\keywords}[1]{\textbf{\textit{Keywords: }} #1}
\newtheorem{theorem}{Theorem}[section]
\newtheorem{lemma}[theorem]{Lemma}
\newtheorem{prop}[theorem]{Proposition}
\newtheorem{cor}[theorem]{Corollary}
\newtheorem{problem}[theorem]{Problem}
\newtheorem{question}[theorem]{Question}
\newtheorem{conjecture}[theorem]{Conjecture}
\newtheorem{claim}[theorem]{Claim}
\newtheorem{defn}[theorem]{Definition} 
\theoremstyle{remark}
\newtheorem{remark}[theorem]{Remark}
\newtheorem{example}[theorem]{Example}
\newtheorem{condenum}{Condition}

\newcommand{\cc}{{\mathbb{C}}}   
\newcommand{\mQ}{{\mathbb{Q}}}   
\newcommand{\ff}{{\mathbb{F}}}  
\newcommand{\nn}{{\mathbb{N}}}   
\newcommand{\qq}{{\mathbb{Q}}}  
\newcommand{\rr}{{\mathbb{R}}}   
\newcommand{\zz}{{\mathbb{Z}}}  
\newcommand{\fp}{{\mathfrak{p}}}
\newcommand{\ra}{{\rightarrow}}
\newcommand{\divides}{\,|\,}

\newcommand\DN[1]{{\color{black} {#1}}}
\newcommand\JK[1]{{\color{red} {#1}}}
\newcommand\oline[1] {{\overline{#1}}}

\def\technion{Department of Mathematics, Technion - Israel Institute of Technology, Haifa, Israel}
\def\weizmann{Department of Mathematics, Weizmann Institute of Science, Rehovot, Israel}

\title{Sylow-conjugate number fields}
\author{Alexander Lubotzky}
\address{\weizmann}
\email{alexander.lubotzky@weizmann.ac.il}
\author{Danny Neftin}
\address{\technion}
\email{dneftin@technion.ac.il}%

\dedicatory{Dedicated to Moshe Jarden}
\begin{abstract}
By a classical result of Neukirch and Uchida, a number field $K$ is determined by the structure of its absolute Galois group $\Gal(K)$. 
We show that $K$ is not determined by the structure of the Sylow subgroups of $\Gal(K)$, answering a question raised by Florian Pop. 
\end{abstract}
\maketitle

\section{Introduction}
Let $K\subseteq \oline \mQ$ be a number field, i.e., a finite extension of the field of rational numbers $\mQ$ embeded in a fixed algebraic closure $\oline \mQ$ of $\mQ$. 
The seminal results of Neukirch and Uchida \cite{Neu,Uch} assert that $K$ is determined by the structure of its absolute Galois group $\Gal(K)=\Gal(\oline \mQ/K)$. 
Namely, if $L$ is a number field with $\Gal(L)$ isomorphic to $\Gal(K)$ as profinite groups, then $L$ is isomorphic to $K$. 
Moreover, already the structure of the maximal prosolvable quotient 
of $\Gal(K)$ determines $K$ \cite{Uch2}.  
Recent works of   Pop--Topaz \cite{PT}, Saidi--Tamagawa \cite{ST}, and others show that even much smaller quotients 
of $\Gal(K)$ determine $K$. 

The natural question whether the structure of the $p$-Sylow subgroups $\Gal(K)_p$ of $\Gal(K)$, where $p$ runs over all primes, already suffices to determine $K$ was raised by Florian Pop. The goal of this note is to show that this is not the case. To state it precisely, let us first define:
\begin{defn}\label{def:sylow-conjugate}
Two number fields $K$ and $L$ are said to be Sylow-conjugate if for every prime $p$, the $p$-Sylow subgroups of $\Gal(K)$ and $\Gal(L)$ are conjugate in $\Gal(\mQ)$. 
\end{defn}
In particular, the absolute Galois groups of two Sylow-conjugate number fields have isomorphic $p$-Sylow subgroups. Furthermore, it is not difficult to see (Corollary \ref{cor:Gal-degree} below) that two Sylow-conjugate number fields $K,L$ have the same degrees $[K:\mQ]=[L:\mQ]$ and the same Galois closure $M$ over $\mQ$. 
Still we will give many examples for which $K$ and $L$ are not isomorphic to each other. 
Here are some such pairs:
\begin{claim}\label{claim:examples}
The following pairs $(K,L)$ are Sylow-conjugate but not isomorphic: 
\begin{enumerate}
\item[(a)] $K=\mQ(\alpha)$ and $L=\mQ(\beta)$, where $\alpha$ and $\beta$, respectively, are the roots of: 
$$
p_7(x) = x^7-7x+3\text{ and }q_7(x)=x^7+14x^4-42x^2-21x+9.  
$$
\item[(b)] $K=\mQ(\alpha)$ and $L=\mQ(\beta)$, where $\alpha$ and $\beta$, respectively, are the roots of: 
$$\begin{array}{rl} 
p_8(x) & = x^8-4x^7-4x^6+26x^5+2x^4-52x^3+31x+1, \text{ and} \\
q_8(x) & = x^8 + 12x^7 + 30x^6 - 108x^5 - 402x^4 + 342x^3 + 1256x^2 -
    687x - 337. 
\end{array}
$$

\end{enumerate}
\end{claim}

In Section \ref{sec:negative}, we will show, using (mainly) group theoretic methods, how one can get many more such examples.  Some of these extensions are solvable, that is, admit a solvable Galois group $\Gal(M/\mQ)$, and some are not. For example in case (a) above the Galois group $\Gal(M/\mQ)$ is the nonsolvable group $\PSL_3(2)$, while in case (b)  it is the solvable group $\GL_2(3)$. 
We will see that the degree $7$ examples in (a) are of  minimal possible degree over $\mQ$, while in the examples in (b) the order of $
\Gal(M/\mQ)$ is $48$ and this is minimal. 
Along the way, we will see that in many cases, Sylow conjugacy implies conjugacy: for example, if $K$ is a solvable extension of prime degree, then it is determined by Sylow-conjugation, see Theorem \ref{thm:prime}.


Finally, Sylow-conjugation of number fields has some similarities with arithmetic equivalence, i.e.\ number fields with the same Dedekind zeta function. We discuss this in Section \ref{sec:art}, showing that they are still very different: neither one implies the other. We shall also see that there exist pairs $(K,L)$ which are both Sylow-conjugate and arithmetically equivalent and still not isomorphic.

This paper is dedicated to Moshe Jarden on his 80th birthday. Moshe is one of the leading figures of the area of Field Arithmetic and the founding father of this school in Israel. His work has had a lasting impact on both of us --
for which we are very grateful. 

We thank Robert Guralnick for valuable discussions. 
This project has received funding from the European Research Council (ERC) under the European Union's Horizon 2020 research and innovation programme (grant agreement No 882751) and from 
the Israel Science Foundation (grant no. 353/21). Both authors gratefully acknowledge the support and hospitality of the Institute for Advanced Study. All computer computations were carried out using MAGMA. 
\section{Nonisomorphic Sylow-conjugate number fields}\label{sec:negative}
\subsection{A group theoretic criterion}\label{sec:gp-theory}
Let us start with some notation and terminology. 
If $G$ is a profinite group, and $p$ is a prime, we will denote its $p$-Sylow subgroup by $G_p$. Subgroups of profinite groups will alway be assumed to be closed. We say that two subgroups $U,V\leq G$ are {\it Sylow-conjugate} if $U_p$ is conjugate to $V_p$ within $G$ for every prime $p$. 
\begin{lemma}\label{lem:quotients}
Let $G$ be a group, $N\lhd G$ a normal subgroup, and $U,V\leq G$ two subgroups containing $N$. Then: \\
(a) $U$ and $V$ are conjugate in $G$ if and only if $U/N$ and $V/N$ are conjugate in $G/N$. \\ 
(b) If $G$ is profinite, $U$ and $V$ are Sylow-conjugate in $G$ if and only if $U/N$ and $V/N$ are Sylow-conjugate in $G/N$.  
\end{lemma}
\begin{proof}
(a) is clear. For (b): 
If $U$ and $V$ are Sylow-conjugate, then clearly so are $U/N$ and $V/N$, since the image of $U_p$ in $G/N$ is a $p$-Sylow subgroup of $U/N$. 
For the converse, first assume $G$ is finite. Since $U_pN/N$ is a $p$-Sylow subgroup of $U/N$, our assumption yields that $U_pN/N$ and $V_pN/N$ are conjugate in $G/N$. Thus $g^{-1}U_pNg = V_pN$ for some $g\in G$. As both $g^{-1}U_pg$ and $V_p$ are $p$-Sylow subgroups of $V_pN$, they are conjugate in $G$ and hence so are $U_p$ and $V_p$.  
This property then extends to profinite groups by a standard inverse limit argument. 
\end{proof}

Similarly, the following lemma is first verified easily for finite groups and then follows to profinite groups. For $U\leq G$, denote by $U^G=\bigcap_{g\in G} U^g$ the {\it core} of $U$ in $G$, that is, the maximal normal subgroup of $G$ contained in $U$. 
\begin{lemma}\label{lem:core-degree}
Let $G$ be a profinite group and $U,V$ two Sylow-conjugate subgroups of $G$. Then  $U^G=V^G$ and $[G:U]=[G:V]$. 
\end{lemma}



To translate these assertions to Sylow-conjugacy of number fields, we recall:
\begin{remark}\label{rem:fields}
Every isomorphism between two fields $K,L\subseteq \oline\mQ$ extends to an automorphism of $\Gal(\mQ)$, so that $K\cong L$ if and only if $\Gal(K)$ and $\Gal(L)$ are conjugate in $\Gal(\mQ)$. Letting $K^{(p)}$  denote the fixed field of a $p$-Sylow subgroup of $\Gal(K)$, it follows that $K$ and $L$  are Sylow-conjugate if and only if $K^{(p)}\cong L^{(p)}$ for all primes $p$. 
\end{remark}
In view of this remark, the above lemmas give:
\begin{prop}\label{prop:quotients}
Let $K$ and $L$ be two number fields, and $M/\mQ$ a Galois extension containing both. 
Let $U:=\Gal(M/K)$ and $V:=\Gal(M/L)$ be subgroups of $G:=\Gal(M/\mQ)$. 
Then: \\
(a) $K$ and $L$ are Sylow-conjugate if and only if $U$ and $V$ are Sylow-conjugate in $G$. \\
(b) $K$ and $L$ are isomorphic if and only if $U$ and $V$ are conjugate in $G$. 
\end{prop}
\begin{proof} (b) is given by Remark \ref{rem:fields} and Lemma \ref{lem:quotients}.(a). 
To see (a), note that by definition $K$ and $L$ are Sylow-conjugate if and only if $\Gal(K)$ and $\Gal(L)$ are. By Lemma \ref{lem:quotients}.(b), this happens if and only if $U$ and $V$ are Sylow-conjugate in $G$. 
\end{proof}
In particular, it follows that:
\begin{cor}\label{cor:Gal-degree}
Let $K$ and $L$ be two Sylow-conjugate number fields. Then $K/\mQ$ and $L/\mQ$ have the same Galois closure and the same degrees. 
\end{cor}
\begin{proof}
Let $M, G, U, V$ be as in Proposition \ref{prop:quotients}, so that $U$ and $V$ are Sylow-conjugate by the proposition. 
Recall that the core $C:=\core_G(U)$ is the largest subgroup of $U$ which is normal in $G$, so that $M^C$ is the Galois closure of $K/\mQ$. 
To show that the normal closure of $K/\mQ$ and $L/\mQ$ coincide, it suffices to show that $C$ and $D:=\core_G(V)$ coincide. 
Since every $p$-Sylow subgroup of $C$ is of the form $U_p\cap C$ and $C\lhd G$, the groups $U_p\cap C$ and $V_p\cap C$ are conjugate $p$-Sylow subgroups of $C$. 
Since $V$ contains a $p$-Sylow subgroup of $C$ for every $p$, it follows that $V$ and hence $D$ contain $C$. 
By symmetry, $C=D$, proving the claim. 

To show that $[K:\mQ]=[L:\mQ]$, note that as $U$ and $V$ are Sylow-conjugate, the largest $p$-powers dividing $|U|$ and $|V|$ coincide for every prime $p$. 
Thus, $|U|=|V|$ and hence $[K:\mQ]=|G|/|U|=|G|/|V|=[L:\mQ]$ as claimed. 
\end{proof}
Finally, Proposition \ref{prop:quotients}  gives the following recipe for producing examples of pairs $K$ and $L$ which are Sylow-conjugate but  not isomorphic. 
\begin{cor}\label{cor:criterion}
Let $G$ be a finite group which appears as the Galois group of a Galois extension $M/\mQ$. 
Assume $U$ and $V$ are Sylow-conjugate subgroups of $G$ that are nonconjugate, and let $K=M^U$ and $L=M^V$. 
Then $K$ and $L$ are nonisomorphic Sylow-conjugate number fields. 
\end{cor}
A well known conjecture asserts that every finite group appears as the Galois group of some Galois extension of $\mQ$.

\subsection{Examples}\label{sec:poly}

One can produce many examples of tuples $(G,U,V)$ satisfying the conditions of Corollary \ref{cor:criterion}. Here are some ways to do so:

\begin{example}
Recall that for a set of primes $\Pi$, a subgroup $H$ of a finite group $G$ is called a $\Pi$-Hall subgroup if all prime divisors of $|H|$ are in $\Pi$, while all prime divisors of $[G:H]$ are not in $\Pi$. Note that two $\Pi$-Hall subgroups $U,V\leq G$ are always Sylow-conjugate by  Sylow's theorem.  
Thus, every group $G$ which appears as a Galois group over $\mQ$ and has two nonconjugate $\Pi$-Hall subgroups fits into Corollary \ref{cor:criterion}.

 For example, one may pick $G=\PSL_2(11)$ and $\Pi=\{2,3\}$. In this case $G$ contains copies $U$ and $V$ of $A_4$ and the Dihedral group $D_6$ (of order $12$), respectively, as $\Pi$-Hall subgroups. The subgroups $U$ and $V$ are clearly nonconjugate in $G$ since they are nonisomorphic. There are many polynomials whose splitting field $M$ has Galois group $\Gal(M/\mQ)\cong \PSL_2(11)$. For example, Malle--Matzat  \cite[Satz 4]{MM85} show that 
 $$\begin{array}{rl}
 f_{11}(t,x) = &  2x^{11} - 2541x^9 - 45254x^8 + 1026201x^7 + 51653448x^6 
+ 900904653x^5  \\ 
& + 8705450754x^4 + 50915146293x^3 + 180040201308x^2 + 355871173680x  \\ 
& + 303064483392 -2t(x^3 + 22x^2 + 165x + 396)^2(x^3 - 22x^2 - 319x-924)
\end{array}
$$ 
has Galois group $G$ over $\mQ(t)$, and hence by Hilbert's irreducibility theorem $f(t_0,x)$ has Galois group $G$ over $\mQ$ for infinitely many values $t_0\in \mQ$. Taking $M$ to be the splitting field of such $f(t_0,x)$ yields examples of a pair $K=M^U$ and $L=M^V$ of nonisomorphic Sylow-conjugate number fields. 

We note that such examples do not exist when $G$ is solvable since a classical result of P.\ Hall \cite[Theorem 9.3.1]{Hall} asserts that $G$ has $\Pi$-Hall subgroups and that these are all conjugate. 
\end{example}

\begin{example}\label{exam:psl11}
In fact $G=\PSL_2(11)$ has two nonconjugate subgroups $U,V$ isomorphic to $A_5$, see Appendix by Feit in \cite{Lang}. These $U,V$ are $\{2,3,5\}$-Hall subgroups of $G$.  Thus, letting $M$ be a splitting field of the above polynomial $f_{11}(2,x)$, we again have that $K=M^U$ and $L=M^V$ are nonisomorphic but Sylow-conjugate. Explicitly, a computation using MAGMA shows that $K$ and $L$ are realized as the root fields of the polynomials: 
$$\begin{array}{rl}
p_{11}(x) = &  x^{11} - 5090x^9 + 181368x^8 + 8224744x^7 - 828043392x^6 +
    28884349472x^5 \\ 
    & - 558216962688x^4 + 6529632151680x^3 -
    46178757504000x^2 \\ & + 182555783258112x - 310931533135872, \text{ and } \\
q_{11}(x) = & x^{11} - 15270x^9 + 586696x^8 + 44022852x^7 - 3226512608x^6 +
    230394820408x^5 \\ & - 12244387399904x^4 + 151382377029664x^3 -
    3610663124873728x^2 \\ & + 20030100743110656x - 31953398556524544.
    \end{array}
    $$
\end{example}

\begin{example}\label{exam:PSLd}
Let $F$ be a finite field whose order $q$ is a power of a prime $p$, let $d\geq 3$ be an integer, and $G=\PSL_d(q)$. 
Let $U$ be the stabilizer of some fixed $1$-dimensional subspace, i.e.\ $U=\Stab_G(\mathbb F_q\cdot e_1)$, where $e_1,\ldots,e_d$ is a basis for $\mathbb F_q^d$. 
Let $V$ be the stabilize of the hyperplane $W$ spanned by $e_2,\ldots,e_d$. 
Then $U$ and $V$ are maximal parabolic subgroups of $G$ which are nonconjugate in $G$ (although they are isomorphic and in fact conjugate under an outer automorphism of $G$). 

We claim that $U$ and $V$ are Sylow-conjugate in $G$. To see this, note that the $p$-Sylow subgroup of $G$ is contained in a Borel subgroup and hence each parabolic subgroup contains a $p$-Sylow subgroup. In particular, $U$ and $V$ contain $p$-Sylow subgroups of $G$ which are necessarily conjugate in $G$. For every other prime $\ell\neq p$, the subgroups $U$ and $V$ contain a common $\ell$-Sylow subgroup of $G$, so that $U$ and $V$ are Sylow-conjugates. 

We note that the groups $\PSL_d(q)$ are known to appear as Galois groups $\Gal(M/\mQ)$ for many pairs $(d,q)$, but not in general. Here is a an especially interesting case: 
\end{example}
\begin{example}
Let $G=\PSL_3(2)$, and $U$ and $V$ be its index-$7$ subgroups from Example \ref{exam:PSLd}. Note that $G\cong \PSL_2(7)$ as abstract groups. 
It is well known that $G$ appears as a Galois group over $\mQ$. Moreover, $K=M^U$ and $L=M^V$ can be chosen to be the fields $\mQ(\alpha)$ and $\mQ(\beta)$, respectively, where $\alpha$ and $\beta$ are the  roots of the pair of polynomials $p_7,q_7$ 
from Claim \ref{claim:examples}.(a) given by Trinks \cite{Trin}, or roots of one of the pairs of polynomials given by Erbach--Fischer--McKay \cite{EFM}, e.g.: 
$$
u_7(x):= x^7 - 154x + 99, \text{ and }v_7(x):=x^7 - 231x^3 - 462x^2 + 77x + 66. 
$$
In Section \ref{sec:positive}, we will show that this is a minimal example in the sense that there is no pair of nonisomorphic Sylow-conjugate fields whose degree over $\mQ$ is less than $7$. 
\end{example}
\begin{example}\label{exam:regular}
Let $U$ and $V$ be two nonisomorphic finite groups satisfying: for every prime $p$, the $p$-Sylow subgroup of $U$ is isomorphic to the $p$-Sylow subgroup of $V$. There are many such examples: e.g.\ $U$ is the cyclic group of order $2\ell$, $\ell$ prime (or any odd number), and $V$ the Dihedral group of order $2\ell$. 

Embed both $U$ and $V$ into $S_n$, for $n=|U|=|V|$, via the regular permutation representation. One can see that for a $p$-Sylow subgroup $P$ of $U$ (or $V$), the permutation representation is a union of $|U|/|P|$ copies of the regular representation of $P$. Thus $U_p$ and $V_p$ are conjugate within $S_n$. It is classical that $S_n$ is the Galois group of many extensions $M/\mQ$, so that the triples $(S_n,U,V)$ give many pairs $K=M^U$ and $L=M^V$ of nonisomorphic Sylow-conjugates. 
\end{example}

Up to now, all of our examples were nonsolvable. One can produce also solvable examples such as the following. 
\begin{example} Let $G=\GL_2(3)$ be a group of order $48=16\times 3$. Letting $P$ be the $3$-Sylow subgroup generated by 
$$
\left[\begin{array}{cc}
 1 & 1 \\
 0 &  1 \\
\end{array}\right].
$$
We let $U:=\langle P,A\rangle$ and $V:=\langle P,B\rangle$, where 
$$ A:=\left[\begin{array}{cc}
 1 & 0 \\
 0 &  -1 \\
\end{array}\right] 
\text{ and }B:=\left[\begin{array}{cc}
 -1 & 0 \\
 0 &  1 \\
\end{array}\right],$$
are conjugate involutions normalizing $P$, so that $U$ and $V$ are Sylow-conjugate. 
We claim that $U$ and $V$ are nonconjugate. 
Indeed, as their $3$-Sylow subgroups coincide and are unique in each of them, if $U$ and $V$ are conjugate, they are conjugate in the normalizer $N:=N_G(P)$ of $P$. 
Since the number of $3$-Sylow subgroups in $G$ is $[G:N_G(P)]=4$ (e.g.\ using the Sylow theorems), we see that $|N_G(P)|=12$, and hence $N_G(P)$ is generated by $A,B$ and $P$. However, $A$ and $B$ are nonconjugate in $\langle A,B,P\rangle=N_G(P)$ (which consists of upper triangular matrices), proving the claim. 

By Shafarevich's theorem, every solvable group appears as the Galois group of a polynomial. The polynomial $p_8$ in Claim \ref{claim:examples} is a well known example of a polynomial with Galois group $\GL_2(3)$ over $\mQ$ with point stabilizer $S_3$.  
A direct computation shows that a root field of the polynomial $q_8$ from the claim is the fixed field of a nonconjugate copy of $S_3$. 
In Section \ref{sec:positive}, we will show that in this example $G$ is of smallest possible order. 
\end{example} 


\section{When Sylow-conjugation implies  isomorphism}\label{sec:positive}
In this section, we state conditions under which Sylow conjugation does imply isomorphism. 
We first show that two Sylow-conjugate number fields $K$ and $L$ of prime degree $p$ over $\mQ$ are ``usually" isomorphic. 
The following result makes use of the classification of finite simple groups (CFSG). 
\begin{theorem}\label{thm:prime}
Let $K,L$ be two Sylow-conjugate number fields of prime degree $p$ over $\mQ$. 
Then $K$ and $L$ are isomorphic unless their common Galois closure  $M/\mQ$ satisfies one of the following:
\begin{enumerate}
\item[(a)]  $p=11$, $\Gal(M/\mQ)\cong \PSL_2(11)$, and $K$ and $L$ are number fields of the type described in Example \ref{exam:psl11} (the fixed fields of the two different  conjugation classes of $A_5$ in $\PSL_2(11)$). 
\item[(b)] There exists a prime $d\geq 3$ and a prime power $q$, such that 
\begin{equation}\label{equ:sum-primes}
p=\frac{q^d-1}{q-1},
\end{equation} 
and $\Gal(M/\mQ)$ is an almost simple group with socle  $\PSL_d(q)$. Here
  $K$ and $L$ are the fixed fields of two nonconjugate subgroups $U'$ and $V'$ in $G$ whose intersection with $\PSL_d(q)$ are equal to the two maximal parabolic subgroups described in Example \ref{exam:PSLd}. 
\end{enumerate}
\end{theorem}
In particular, the following is the direct consequence for solvable groups, which however does not require the classification: 
\begin{cor}
If $K$ and $L$ are solvable (i.e.\ $\Gal(M/\mQ)$ is solvable) Sylow-conjugate extensions of prime degree, then they are isomorphic. 
\end{cor}
\begin{proof}[Proof of Theorem \ref{thm:prime}]
Let $G:=\Gal(M/\mQ)$, and $U,V\leq G$ be the index-$p$ subgroups fixing $K$ and $L$, respectively. 
Then $G$ acts faithfully on $G/U$ (and on $G/V$) as a degree $p$ permutation group. 
As $p$ is prime,  Burnside's theorem \cite{Mul} shows that either (i) $G$ is solvable, in which case it is a subgroup of $\Aff_1(\mathbb F_p)=\mathbb F_p\rtimes \mathbb F_p^\times$, or (ii) $G$ is doubly transitive, and hence almost simple \cite[Theorem 7.2E]{DM}. 

In case (i), $G$ is isomorphic to $\mathbb F_p\rtimes C$ with $|C|\divides p-1$ and in particular $|C|$ is prime to $p$. 
Hence $\HG^1(C,\mathbb F_p)=0$ and $\mathbb F_p$ has a unique complement under conjugation. 
Thus, every two subgroups of $G$ of index $p$ are conjugate. 

In case (ii), let $S$ be the socle of $G$, and note that since $U$ is of index $p$ and has trivial core, $S\cap U$ is of index $p$ in $S$. We next append to the work of Guralnick \cite{Gur} who classifies, using CFSG, all the finite simple groups $S$ with a subgroup of prime power index. The cases where $S$ has a subgroup of index $p$ are: (a') $S\cong \PSL_2(11)$ with $p=11$; (b') $S\cong \PSL_d(q)$ in the action described in (b) with $p=(q^d-1)/(q-1)$; (c) the alternating group $S\cong A_p$ for arbitrary prime $p$; and (d) $S\cong M_{23}$ with $p=23$ or $S\cong M_{11}$ with $p=11$. 
In (c), the index-$p$ subgroups of the alternating group $A_p$ (resp.\ of the symmetric group $S_p=\Aut(A_p)$) are conjugate. Similarly, the index-$p$ subgroups are conjugate for the Mathieu groups in (d) (which are the unique almost simple groups with the given socle). 
 In (b'), $G$ is as in (b) and $d$ is prime as well \cite{Gur}. Note that $d\geq 3$ since for $d=2$ the index-$p$ subgroups are conjugate. 
 In (a'), the degree-$11$ action of $\PSL_2(11)$ does not extend to a degree-$11$ action of $\PGL_2(11)$, and hence $G=S\cong \PSL_2(11)$ as in (a).  
\end{proof}
\begin{remark}
Note that case (a) indeed gives a counterexample as described in Example \ref{exam:psl11}. Also if for a prime $p$,   \eqref{equ:sum-primes}
has a solution with $d\geq 3$ and a prime power $q$, then indeed $G=\PSL_d(q)$ has two subgroups $U$ and $V$ of index $p$ which are Sylow-conjugate but nonconjugate, see Example \ref{exam:PSLd} above. 

What is unknown is whether $G=\PSL_d(q)$ is indeed always a Galois group over $\mQ$. It is also unknown if there are infinitely many primes for which \eqref{equ:sum-primes} has such solutions $q$ and $d$, see \cite{EGSS} for a discussion. 
But certainly there are some, for example: 
$$ 13 = \frac{3^3-1}{3-1},\, 31 = \frac{5^3-1}{5-1},\,  73 = \frac{8^3-1}{8-1},\text{ and } 1772893 = \frac{11^9-1}{11^3-1}. $$
\end{remark}

We  note that in the class of nilpotent groups Sylow conjugacy implies conjugacy:
\begin{lemma} Let $K$ and $L$ be Sylow-conjugate number fields whose common Galois closure $M/\mQ$ has a nilpotent Galois group $G$. Then $K$ and $L$ are isomorphic.
\end{lemma}
\begin{proof}
As $G$ nilpotent, it is the product $\prod_p G_p$ of its $p$-Sylow subgroups, and hence $U_p=V_p^{x_p}$ are already conjugate by an element $x_p\in G_p$, for every prime $p$. Since the elements $x_p$, $p$ prime, commute, it follows that $U=V^x$, where $x=\prod_p x_p$. 
\end{proof}
We give one more infinite family for which Sylow-conjugacy implies conjugacy.
\begin{prop}
Suppose $G=A\rtimes H$, where $A$ is an abelian group which is irreducible as an $H$-module. 
Let $M/\mQ$ be a $G$-extension and $K$ and $L$ be Sylow-conjugate subfields fixed by complements of $A$ in $G$. 
Then $K$ and $L$ are isomorphic. 
\end{prop}
\begin{proof}
Let $U=\Gal(M/K)$ and $V=\Gal(M/L)$ be the corresponding complements of $A$ in $G$. Since $A$ is an irreducible $H$-module, $A$ is an elementary abelian $p$-group for some prime $p$. Let $\pi:G\ra H$ the natural projection modulo $A$. 

We first claim that  the restriction map 
$
\res_p:\HG^1(U,A)\ra \HG^1(U_p,A)
$
is injective. Indeed, letting $\cores_p:\HG^1(U_p,A)\ra \HG^1(U,A)$ denote the correstriction map, it is well known that $\cores_p\circ \res_p$ is the multiplication-by-$[U:U_p]$ map. Since $[U:U_p]$ is coprime to $|A|$, this multiplication map is an isomorphism, so that $\res_p$ is injective. 



Since $U$ and $V$ are complements of $A$ in $G$ there is a cocycle $\chi\in \ZG^1(U,A)$ 
for which the homomorphism $f_\chi:U\ra G$, $u\mapsto   u\cdot \chi(u)$ maps $U$ isomorphically to $V$. 
The subgroup $f_\chi(U_p)$ is a $p$-Sylow subgroup of $V$, which we denote by $V_p$. 
Note that since $\Image\chi\in A$, we have $U_pA=V_pA$.  Let $H_p:=\pi(U_p)=\pi(V_p)$. 

Since $U$ and $V$ are Sylow-conjugate, there exists $g\in G$ such that $U_p^g=V_p$. 
Thus $\pi(g)$ is in the normalizer $N_H(H_p)$  of $H_p$ in $H$, so that $g\in \pi^{-1}(N_H(H_p))=N_{U}(U_p)A$, where the latter equality holds since $\pi$ maps $N_U(U_p)$  isomorphically to $N_H(H_p)$  (as $U$ and $U_p$ are mapped isomorphically to $H$ and $H_p$, respectively). 
Writing $g=na$ for $n\in N_{U}(U_p)$ and $a\in A$, we see that $V_p=U_p^{na}=(U_p^{n})^{a}=U_p^{a}$, so that $U_p$ and $V_p$ are conjugate by $a\in A$. 

We claim that the latter implies that the restriction $\chi_p:=\res_p(\chi)$ has a trivial class in $\HG^1(U_p,A)$. Indeed, since $U_p$ and $V_p$ are conjugate in $U_pA=V_pA$, by composing $f_{\chi_p}:U_p\ra V_p$ with inner conjugation by $a^{-1}$, we obtain a map $f_{\chi'_p}:U_p\ra U_p$ for some $\chi'_p\in \ZG^1(U_p,A)$ cohomologically equivalent to $\chi_p$. 
Thus $f_{\chi'_p}(u)=u\chi'_p(u)\in U_p$, so that $f_{\chi'_p}(u)u^{-1}\in A\cap U=1$, that is, $f_{\chi'_p}(u)=u$ and $\chi'_p(u)=1$ for all $u\in U_p$. 

Finally, since by the above claim $\res_p$ is injective, this implies $[\chi]\in \HG^1(U,A)$ is trivial, and hence that $U$ and $V$ are conjugate as desired. 
\end{proof}

Finally, we show that Sylow-conjugate number fields which are ``smaller" than those in Claim \ref{claim:examples} are isomorphic. 

\begin{prop}
Let $K$ and $L$ be two Sylow-conjugate number fields with common Galois closure $M/\mQ$ and Galois group $G=\Gal(M/\mQ)$. 
Assume  that either $[K:\mQ]=[L:\mQ]\leq 6$ or that $|G|<48$. 
Then $K$ and $L$ are isomorphic. 
\end{prop}
\begin{proof}
Set $d:=[K:\mQ]=[L:\mQ]$ and first assume $d\leq 6$. 
By Proposition \ref{prop:quotients} and Lemma \ref{lem:core-degree}, the subgroups $U=\Gal(M/K)$ and $V=\Gal(M/L)$ are Sylow-conjugate subgroups of $G$ of index $d$ with trivial core.  In particular, we may identify $G$ with a subgroup of $S_d$. 
We claim that $U$ and $V$ are conjugate. We checked this using MAGMA, as well as analyzed by hand as follows:

First note that if the order of $U$ and $V$ is a power of a prime, as they are Sylow-conjugate, they are conjugate. 
Henceforth,  assume 
$|U|,|V|$ are not prime powers. 

For $d\leq 2$, one has $U,V\lhd G$ and hence $U=V$. 
For $d=3$, since $U,V\leq G$ have trivial core,  $G=S_3$, and $U,V$ are $2$-Sylow subgroups of $G$. As $U,V$ are Sylow-conjugate they are conjugate. For $d=4$: if $G=S_4$, it has a unique conjugacy class of index $4$ subgroups. For $G\lneq S_4$, as $d=4$, it follows that $|U|,|V|$ are prime powers. 
The case $d=5$ is covered by Theorem \ref{thm:prime} as $d$ is a prime. 

The case $d=6$ is more interesting:
If $G=S_6$ or $A_6$, then $G$ indeed has two different conjugacy classes of index $6$ subgroups. One is $U=S_5$ (resp.\ $U=A_5$) and the second is the image $V$ of the action $S_5\ra \Sym\{P_1,\ldots,P_6\}=S_6$ of $S_5$ on its six $5$-Sylow subgroups $P_1,\ldots,P_6$. 
Indeed, $U$ and $V$ are nonconjugate in $S_6$ (resp.\ $A_6$) since $U$ fixes a point while $V$ does not. 
But at the same time, the $3$-Sylow subgroups $U_3$ and $V_3$ of $U$ and $V$, which are of order $3$, are nonconjugate as well since $U_3$ has a fixed point while $V_3$ does not. 

Henceforth assume $G\lneq S_6$ is a proper subgroup other than $A_6$. 
The maximal subgroups $G\leq S_6$ (resp.\ $G\leq A_6$) satisfying the above are of order $48$, $72$, and $120$ (resp.\ $24$,  $36$, and $60$).  

The only subgroup of order $120$ (resp.\ $G\leq A_6$ of order $60$)  is $S_5$ (resp.\ $A_5$). 
If $G=S_5$ (resp.\ $A_5$), it has a unique conjugacy class of index $6$ subgroups, namely the Frobenius group of order $20$ (resp.\ Dihedral group of order $10$) normalizing a $5$-Sylow subgroup. 
If $G\lneq S_5$ (resp.\ $G\lneq A_5$) is a proper subgroup other than $A_5$, then $|G|/6$ is a prime power, so that in this case as well $U$ and $V$ are conjugate. 

The only subgroup of $S_6$ of order  $72$ (resp.\ of $A_6$ of order $36$) is the stabilizer of a partition into two blocks of size $3$, 
  so that here we may assume $G\leq S_3\wr C_2=(S_3\times S_3)\rtimes C_2$. 
The only such subgroups for which $|G|/6$ is not a prime power are $S_3\wr C_2$ and its subgroups of order $36$. 
Letting $G$ be one of those groups, 
 $G_3=C_3\times C_3$ is normal in $G$, and hence the subgroup $U_3=G_3\cap U$ of order $3$ is normal  $U_3\lhd U$. Since $U_3$ is normalized by $U$, a subgroup of index $6$, and by $G_3$,  the normalizer $N_G(U_3)$  is of index $\leq 2$ in $G$. 
 If $N_G(U_3)=G$, i.e.\ $U_3\lhd G$, then $V_3=U_3$. In this case, since $U_2=V_2^x$ for some $x\in G$, one has $V^x=V_3^xV_2^x=U_3U_2=U$. 
 Otherwise, $[G:N_G(U_3)]=2$, and $U_2$ is also a $2$-Sylow subgroup of $N_G(U_3)$. Similarly $V_2$ is a $2$-Sylow of $N_G(V_3)$. 
Moreover, as $V_3^x=U_3$ for some $x\in G$, one has  $N_G(V_3)^x=N_G(U_3)$, and so $V_2^x$ is also a $2$-Sylow subgroup of $N_G(U_3)$. Thus $V_2^{xy}=U_2$ for some $y\in N_G(U_3)$, so that 
$$V^{xy}=V_2^{xy}V_3^{xy}=U_2U_3^y=U_2U_3=U.$$
 
Finally if $G\leq S_6$ is of order $48$ (resp.\ $G\leq A_6$ is of order $24$) or a subgroup of it, then $|G|/6$ is of prime power order, completing the proof in case $d\leq 6$. 

Assume $|G|<48$ (and $d$ is arbitrary), and $U,V\leq G$ are Sylow-conjugate. 
Note that when  $d:=[G:U]=[G:V]$ is either prime or of degree $4$ or $6$, then the claim follows from Theorem \ref{thm:prime} and the first assertion of the proposition. Thus to deduce the second assertion, it suffices to note that if $|G|<48$ and $[G:U]=[G:V]\geq 8$ is not a prime, then $|U|,|V|$ are  prime powers, so that $U$ and $V$ are conjugate. 

\end{proof}
\section{Between Sylow-conjugacy and arithmetic equivalence}\label{sec:art}
Two number fields $K$ and $L$ are said to be arithmetically equivalent if their Dedekind zeta function are equal $\zeta_K(s)=\zeta_L(s)$. It is well known \cite{Per} that this happens if and only if $K$ and $L$ have a common Galois closure $M/\mQ$ satisfying the following: The subgroups $U=\Gal(M/K)$ and $V=\Gal(M/L)$ of $G=\Gal(M/\mQ)$ satisfy $|C\cap U|=|C\cap V|$ for every conjugacy class $C$ of $G$. 
So arithmetic equivalence, just like Sylow conjugation is a weak form of conjugation. This and other\footnote{We  show above that Sylow-conjugacy implies conjugacy if the degree is $<7$ or, in case $G$ is solvable, if the degree is a prime. Similar results are also proved by Perlis for arithmetic equivalence: The case of degree $<7$  is given in  \cite[Theorem 3]{Per}, while the case where $G$ is solvable and the degree is a prime $p$ is covered by \cite[Theorem 2(g)]{Per} (since stabilizers are of order coprime to $p$).} similarities may suggest that the two properties are equivalent. 
In what follows we will show that this is not the case (Examples \ref{exam:A} and \ref{exam:B}). 
Example \ref{exam:C} will show that there are number fields which are Sylow-conjugate as well as arithmetically equivalent and still not isomorphic. 
\begin{example}\label{exam:A}
Let $p$ be an odd prime, $U=C_{2p}$ the cyclic group of order $p$, and $V=D_{2p}$ the Dihedral group of order $2p$. Embed them regularly in $\Sym(2p)$. Every element of $C_{2p}$ (resp.\ $D_{2p}$) of order $2$ gives rise to a product of $p$ distinct transpositions, and an element of order $p$ inducs a product of two disjoint $p$-cycles. Now, $U$ and $V$ are therefore Sylow-conjugate in $\Sym(2p)$, but are clearly not arithmetically equivalent, as $C_{2p}$ has an element of order $2p$ but $D_{2p}$ does not. As $\Sym(2p)$ is a Galois group over $\mQ$, this induces (as in Example \ref{exam:regular}) pairs of number fields which are Sylow-conjugate but not arithmetically equivalent. 
\end{example}
\begin{example}\label{exam:B}
Let $p$ be a prime, $U=C_p\times C_p\times C_p$ and 
$$ V = \left\{ \left[\begin{array}{ccc}
 1 & a & b \\
 0 &  1 & c \\
 0& 0 & 1 
\end{array}\right]\,|\, a,b,c\in \mathbb F_p\right\}$$ the Heisenberg group over the field $\mathbb F_p$. 
Embed $U$ and $V$ regularly into $\Sym(p^3)$. 
Every nontrivial element of $U$ (resp.\ $V$) when acting on $U$ (resp.\ $V$) induces a permutation which is a product of $p^2$ disjoint $p$-cycles. 
Then by the above criterion, $U$ and $V$ are arithmetically equivalent, but they are clearly not Sylow-conjugate, as they are nonisomorphic $p$-groups. 
\end{example}
\begin{example}\label{exam:C}
Let us take a second look at Example \ref{exam:PSLd}. We showed that these two maximal parabolic subgroups $U,V\leq G$ are Sylow-conjugate  and nonconjugate. We claim now that the examples are also arithmetically equivalent. Indeed, the above group theoretic criterion is equivalent to saying that the linear representation of $G$ on $\mathbb C[G/U]$  is isomorphic to the one  on $\mathbb C[G/V]$. (Note that $U$ and $V$ are conjugate if and only if the permutational representations are isomorphic). 

To show that the two linear representations are isomorphic it suffices to show that $\tr(g_{|\mathbb C[G/U]})=\tr(g_{|\mathbb C[G/V]})$ for every $g\in G$. Next note that:
\begin{enumerate}
\item[(a)] In the linear representation induced by a permutational representation $\tr(g)$ is equal to the number of fixed points. 
\item[(b)] In our case the action of $G$ on $G/U$ is equivalent to the action of $G$ on the $1$-dimensional subspaces of $\mathbb F_q^d$, while the action on $G/V$ is equivalent to that on hyperplanes. 
\end{enumerate}
Whenever $g\in G$ preserves a $1$-dimensional subspace $L$, its transpose $g^t$ preserves the hyperplane $L^\perp$ 
perpendicular to $L$. Thus, $\tr(g_{|\mathbb C[G/U]})=\tr(g^t_{|\mathbb C[G/V]})$. But $g$ and $g^t$ are conjugate in $\PSL_d(q)$ and hence $\tr(g_{|\mathbb C[G/U]})=\tr(g_{\mathbb C[G/V]})$, and $U$ and $V$ give rise to arithmetically equivalent  nonisomorphic number fields, provided $\PSL_d(q)$ is a Galois group. This is the case at least for $\PSL_3(2)$, which gives the examples in Claim \ref{claim:examples}.(b). 
\end{example}

\end{document}